\documentclass[11pt,leqno]{article}
\usepackage{graphicx}

\usepackage{epstopdf,epsfig}
\usepackage[colorlinks=false]{hyperref}

\usepackage{amsmath,amsthm,mathtools, amsfonts,amssymb,latexsym,amscd,enumerate}

\usepackage{palatino}
\usepackage[mathcal]{euler}
 
\usepackage{xcolor}

\usepackage{xy}
\xyoption{all}

\swapnumbers

\newcommand{\morphism}{\rightsquigarrow}
 
\newcommand{\evtl}{\mathrm{evtl}}

\numberwithin{equation}{section}

\theoremstyle{plain}

\newtheorem{theorem}{Theorem}[section]
\newtheorem{tech-theorem}[theorem]{Technical Theorem} 

\newtheorem{lemma}[theorem]{Lemma}
\newtheorem*{theorem*}{Theorem}

\newtheorem{corollary}[theorem]{Corollary}
\newtheorem*{lemma*}{Lemma}

\theoremstyle{definition}

\newtheorem*{claim*}{Claim}
\newtheorem*{definition*}{Definition}
\newtheorem{definition}[theorem]{Definition}

\newtheorem{num}[theorem]{}

\theoremstyle{remark}
\newtheorem{remark}[theorem]{Remark}

\newcommand{\sig}{\operatorname{Sgn}}

\newcommand{\cpt}{\text{cpt}}

\newcommand{\id}{\textup{Id}}

\newcommand{\R}{\mathbb R}
\newcommand{\C}{\mathbb{C}}

\begin{document}
	
\title{C*-Algebraic Higher  Signatures and an Invariance Theorem in Codimension Two}

\author{Nigel Higson, Thomas Schick, and Zhizhang Xie}

\date{}

\maketitle

\begin{abstract}
\noindent
We revisit the construction of signature classes in $C^*$-algebra $K$-theory, and develop a variation that allows us to prove equality of signature classes in  some situations involving homotopy equivalences of noncompact manifolds that are only defined outside of a compact set.  As an application, we prove a counterpart for signature classes of a codimension two vanishing theorem for the index of the Dirac operator on spin manifolds (the latter  is due to   Hanke, Pape and Schick, and is a development of well-known work of Gromov and Lawson).
\end{abstract}

\section{Introduction}

Let $X$ be a connected, smooth, closed and oriented manifold of dimension $d$, and let  $\pi$ be a discrete group.  Associated to each  $\pi$-principal bundle  $\widetilde X$ over $X$ there is a \emph{signature class} 
\[
\sig_{\pi} (\widetilde X/X) \in K_d(C^*_r(\pi))
\]
in the topological $K$-theory of the reduced group $C^*$-algebra of $\pi$. The   signature class can be defined either index-theoretically or combinatorially, using a triangulation.  Its main property, which is easiest to see from the combinatorial point of view, is that if 
$
h\colon Y \to X
$
is a homotopy equivalence of smooth, closed, oriented manifolds that is covered by a
map of $\pi$-principal bundles $\widetilde Y\to\widetilde X$, then
\begin{equation}
\label{eq-htpy-invariance}
\sig_{\pi}(\widetilde X/X) = \sig_{\pi}(\widetilde Y/Y) \in K_d \bigl (C^*_r(\pi)\bigr ) .
\end{equation}
This is the point of departure for the $C^*$-algebraic attack on Novikov's
higher signature conjecture (see \cite{Rosenberg93,Rosenberg16} for surveys). 

This note has two  purposes. The first  is to simplify some of the machinery involved in the combinatorial approach to   signature classes in $C^*$-algebra $K$-theory.  The main novelty is a technique to construct   signature classes that uses only finitely-generated and projective modules (in part by appropriately adjusting the $C^*$-algebras that are involved).  

The second purpose is to prove a codimension-two invariance result for the signature that is  analogous to a recent theorem of Hanke, Pape and Schick \cite{HankePapeSchick15} about positive scalar curvature (which is in turn developed from well-known work of Gromov and Lawson \cite[Theorem 7.5]{MGBL83}).  To prove the theorem we shall use generalizations of the signature within the context of coarse geometry. These have previously been  studied closely (see for example \cite{Roe96}), but our new approach to the signature has some definite advantages in this context (indeed it was designed with coarse geometry in mind).
Here is the result. Let
\[ 
h: N \longrightarrow M
\]
be a smooth, orientation-preserving homotopy equivalence  between two smooth, closed, oriented manifolds of dimension $d{+}2$. 
Let $X$ be a smooth, closed, oriented submanifold of $M$ of codimension $2$ (and hence dimension $d$), and assume that 
$ h $
is transverse to $X$, so that 
the inverse image 
\[ 
Y = h^{-1}[X] 
\]
is a smooth, closed, oriented submanifold of $N$ of  codimension $2$. Suppose $\widetilde X$ is the universal cover of $X$. Then it pulls back along the map $h$ to give a $\pi_1(X)$-covering space $\widetilde Y$  of $Y$.

\begin{theorem} 
\label{thm-main-theorem}
Assume that 
\begin{enumerate}[\rm (i)]
\item $\pi_1(X) \to \pi_1(M)$ is injective;
\item  $\pi_2(X)\to \pi_2(M)$ is surjective; and 
\item the normal bundle of $X$ in $M$ is trivializable.
\end{enumerate}
Then 
  \begin{equation}\label{eq:main_eq}
2 \bigl (\sig_{\pi_1(X)}(\widetilde X/X) -  \sig_{\pi_1(X)}(\widetilde
Y/Y) \bigr ) = 0 \quad \text{in $K_{d}(C ^*_r(\pi_1(X)))$.}
\end{equation}
\end{theorem}
     
Let us make some remarks about this theorem.  First, we do not know whether the factor of $2$ in \eqref{eq:main_eq} is
    necessary,  but in view of what is known about   $L$-theory concerning
     splitting obstructions and ``change of decoration,'' where non-trivial
    $2$-torsion phenomena do occur,  we suspect that there are situations where
     the factor of $2$ in \eqref{eq:main_eq} is indeed necessary.

    In surgery theory (symmetric) signatures  are
    defined  in the $L$-theory groups of the group ring $\mathbb{Z}[\pi]$. There
    are comparison maps from these $L$-groups to our $K$-groups (at least after
    inversion of $2$; see \cite{Land-Nikolaus} for a thorough discussion
    of this issue) and the comparison maps take the respective signatures to each other (this correspondence is developed systematically in  \cite{HigsonRoeSurgToAnal1,HigsonRoeSurgToAnal2,MR2220524}). We leave it as an open question to lift Theorem~ \ref{thm-main-theorem} to an equality of $L$-theoretic higher signatures.

    Finally, in view of the strong Novikov conjecture, one might expect that all
    interesting homotopy invariants of a closed oriented manifold $M$ with
    universal cover $\widetilde M$ can be
   derived from the signature class 
   \[
   \sig_{\pi_1(M)}(\widetilde M/M)\in K_*(C^*_r(\pi_1(M))).
   \]
    Theorem
    \ref{thm-main-theorem} shows that $2\sig_{\pi_1(X)}(X)\in
    K_*(C_r^*(\pi_1(X)))$ is a new homotopy invariant of the ambient manifold
    $M$, under the conditions of the theorem. It is an interesting
    question to clarify the relationship between  this new invariant and the signature class of $\widetilde M / M$.

Here is a brief outline of the paper.
In Section~\ref{sec:sign-prop} we shall summarize the key properties of the   signature class, including a new invariance property that is required for the proof of Theorem~\ref{thm-main-theorem}.  In Section~\ref{sec-proof} we shall present the proof of Theorem~\ref{thm-main-theorem}.  As in the work of Hanke, Pape and Schick, the proof involves coarse-geometric
ideas, mostly in the form of the partitioned manifold index theorem of Roe
\cite{Roe88}.  With our new approach to the signature, the proof of the
positive scalar curvature result can be adapted with very few changes.
  In the remaining sections we shall give our variation on the construction of the signature class, using as a starting point the approach of Higson and Roe in  \cite{HigsonRoeSurgToAnal1,HigsonRoeSurgToAnal2}.

\section{Properties of the C*-Algebraic Signature}
\label{sec:sign-prop}

In this section we shall review some  features of the $C^*$-algebraic signature, and then introduce a new property of the signature  that we shall use in the proof of Theorem~\ref{thm-main-theorem}.  We shall not give the definition of the signature in this section, but we shall  say more later, in Section~\ref{sec-controlled-signature}. 

\begin{num}
\label{num-higher-signature}
Let $\pi$ be a discrete group and let $X$ be a smooth, closed and oriented manifold of dimension $d$ (without boundary).  Mishchenko, Kasparov and others   have associated to every  principal $\pi$-bundle $\widetilde X$ over $X$ a \emph{signature class} in the topological $K$-theory of the reduced  $C^*$-algebra of the group $\pi$: 
 \begin{equation}
 \label{eq-signature-class}
 \sig_{\pi}(\widetilde X/X) \in K_d (C^*_r(\pi)).
 \end{equation}
\end{num}

\begin{num}
\label{num-homotopy-invariant}
The first and most important property of the signature class is that it is an \emph{oriented homotopy invariant}: given a commuting diagram of principal $\pi$-fibrations
 \[
 \xymatrix{ \widetilde Y \ar[r] \ar[d] & \widetilde X \ar[d] \\
 Y \ar[r] _h & X
 }
 \]
in which  the bottom map is an orientation-preserving homotopy equivalence, we have  
 \[
  \sig_{\pi}(\widetilde Y / Y)  =  \sig_{\pi}(\widetilde X / X)  \in K_d (C^*_r(\pi)).
 \]
 \end{num}
 
\begin{num}
Assume that  $X$ is in addition connected, and let  $\widetilde X$ be
{a} universal covering manifold.\footnote{It is not  necessary to keep
  track of basepoints in what follows since inner automorphisms of $\pi$  act
  trivially on $K_*(C^*_r(\pi))$.} In this case we shall abbreviate our
notation and write
 \[
  \sig_{\pi_1(X)}(X) \in K_d (C^*(\pi_1(X))),
 \]
omitting mention of the cover.  If $d$ is congruent to zero, mod four, then we  can recover  the usual signature of the base manifold $X$  from the $C^*$-algebraic signature $   \sig_{\pi_1(X)}(X) $ by applying the homomorphism
 \[
  \operatorname{Trace}_\pi\colon K_d (C^*_r(\pi)) \longrightarrow \R
  \]
 associated to the canonical trace on $C^*_r (\pi)$ as a consequence of
 Atiyah's $L^2$-index theorem.   But the $C^*$-algebraic signature usually contains much more information.  It  is conjectured to determine the so-called higher signatures of $X$ \cite{MR0292913, MR0362407, MR1388299}. This is the  \emph{Strong Novikov conjecture}, and it  is known in a great many cases,  for example for all fundamental groups of complete, nonpositively curved manifolds \cite{GK88}, for all groups of finite asymptotic dimension \cite{GY98}, and for all linear groups \cite{MR2217050}. 
\end{num}

\begin{num}
 Let us return to general covers. The signature class is   functorial, in the sense that if $\alpha \colon \pi_1  \to \pi_2$ is an injective group homomorphism, and if  the $\pi_1$- and $\pi_2$-principal bundles $\widetilde X_1$ and $\widetilde X_2$ are linked by a    commuting diagram 
  \[
\xymatrix{
\widetilde X_1 \ar[r] \ar[d]  & \widetilde X_2 \ar[d] \\
X \ar@{=}[r] & X 
}
\]
in which the top map is $\pi_1$-equivariant, where $\pi_1$ acts on $\widetilde X_2$ via $\alpha$, then 
\[
 \alpha_*   \colon \sig_{\pi_1} (\widetilde X_1 / X) \longmapsto \sig_{\pi_2}(\widetilde X_2 / X)  \in K_d (C^*_r(\pi_2))  .
\]
(We recall here that the reduced group $C^*$-algebra  is functorial for injective group homomorphisms, and therefore so is its $K$-theory.\footnote{It is also possible to define a signature class in the $K$-theory  of the \emph{full} group $C^*$-algebra, and this version is functorial for arbitrary group homomorphisms.  But since it  is in certain respects more awkward to manipulate with we shall mostly confine our attention in this paper to the signature class for the reduced $C^*$-algebra.}) 
\end{num}

\begin{num}
\label{num-bordism-invariant}
The signature is   a \emph{bordism invariant} in the obvious sense: if $W$ is
a compact oriented smooth manifold with boundary, and if $\widetilde W$ is a
principal $\pi$-bundle over $W$, then the signature class associated to the
restriction of $\widetilde W$ to $\partial W$ is zero; see \cite[Section 4.2]{HigsonRoeSurgToAnal2}.
\end{num}

\begin{num}
A locally compact space $Z$ is \emph{continuously controlled over $\R$} if it is equipped with a proper and continuous map 
\[
c \colon Z \longrightarrow \R .
\]
If $Z$ is a smooth and oriented manifold of dimension $d{+}1$  that is continuously controlled over  $\R$, and if   the  control map $c\colon Z \to \R$  is smooth, then the inverse image
\[
X = c^{-1}[a] \subseteq Z
\]
 of any regular value $a\in \R$ is a smooth, closed, oriented manifold.  Given a $\pi$-principal bundle $\widetilde Z$ over $Z$, we can restrict it to $X$ and then  form the signature class 
\eqref{eq-signature-class}.  Thanks to the bordism invariance of the signature class, the $K$-theory class we obtain is independent of the choice of regular value $a\in \R$.  Indeed it depends only on the homotopy class of the proper map $c$.

\begin{definition*}
We shall call the signature class of $X$ above  the \emph{transverse signature class} of the continuously controlled manifold $Z$, and denote it by
\[
\sig_{c,\pi} (\widetilde Z/Z) : = \sig_{\pi}(\widetilde X/X) \in K_d(C^*_r(\pi)) .
\]
\end{definition*}

  \end{num}
  
  \begin{num}
  \label{number-controlled-signature1}
To prove the main theorem of the paper we shall need to invoke an invariance property of the transverse signature class of the following sort. Suppose given a morphism 
\[
h\colon Z \longrightarrow W
\]
of smooth, oriented manifolds that are continuously controlled over $\R$, and suppose that $h$ is, in  suitable sense,  a controlled homotopy equivalence. The conclusion is that if $Z$ and $W$ are equipped with $\pi$-principal bundles that are compatible with the equivalence, then (up to a factor of $2$ that we shall make precise in a moment) the associated  transverse signature classes of $Z$ and $W$ are equal.

Actually we shall need to consider situations where $h$ will be only defined
on the complement of a compact set; see Paragraph~\ref{number-eventual-defs}.
But  for now we shall ignore this additional complication.

 \begin{definition*}
 Let $Z$ and $W$ be locally compact  spaces that are continuously controlled over $\R$, with control maps $c_Z$ and $c_W$. A continuous and proper map $f\colon Z\to W$ is \emph{continuously controlled over $\R$} if  the symmetric difference 
 \[
 c_Z^{-1}[0,\infty)\,\, \triangle\,\,  f^{-1}  c_W^{-1}[0,\infty )
  \]
 is a compact subset of $Z$.
 \end{definition*}
 
Thus the map $f\colon Z\to W $  is  continuously controlled over $\R$ if it is compatible with  the decompositions of $Z$ and $W$ into positive and negative parts (according to the values of the control maps),  modulo compact sets.
 
 \begin{definition*}
 Let $Z$ and $W$ be complete Riemannian manifolds, both  with bounded geometry \cite[Section 1]{Attie94}.  

\begin{enumerate}[\rm (a)]
\item  A \emph{boundedly controlled homotopy equivalence} from $Z$ to $W$ is a  pair of smooth maps 
 \[
 f\colon Z \longrightarrow W \quad \text{and}\quad g \colon W \longrightarrow Z ,
 \]
together with a pair of smooth homotopies 
\[
H_Z \colon Z\times [0,1]\longrightarrow Z 
\quad \text{and}\quad H_W \colon W\times [0,1] \longrightarrow W
\]
 between the compositions of $f$ with $g$ and the identity maps on $Z$ and $W$,
for which all the maps have bounded derivatives,   meaning that, for example,
$\sup_{X\in TZ, ||X||\le 1} ||Df(X)||<\infty$.   

\item If $Z$ and $W$ are continuously controlled over $\R$, then the above
boundedly controlled homotopy equivalence  is in addition \emph{continuously controlled over $\R$} if all the above maps are continuously controlled over $\R$.

\item The above  {boundedly controlled homotopy equivalence} is \emph{compatible} with given  $\pi$-principal bundles over $Z$ and $W$ if all the above maps lift to bundle maps. 
\end{enumerate}

 \end{definition*}
 
The following result is a small extension of Roe's partitioned manifold index theorem  \cite{Roe88}.  We shall prove it in Section~\ref{sec-transverse-signature}.
   
\begin{theorem}
\label{thm-bounded-homotopy-invariance}
 Let $Z$ and $W$ be $(d{+}1)$-dimensional, complete {oriented} Riemannian manifolds  with bounded geometry  that are continuously controlled over $\R$.  If  $\widetilde Z$ and $\widetilde W$ are principal $\pi$-bundles over $Z$ and $W$, and if there is a 
 boundedly controlled {orientation-preserving} homotopy equivalence between $Z$ and $W$ that is continuously controlled over $\R$ and compatible with $\widetilde Z$ and $\widetilde W$, then 
\[
2^\varepsilon \big(\sig _{c,\pi}  (\widetilde Z/ Z) - \sig _{c,\pi}  (\widetilde W/ W) \big) = 0 
\]
in $K_d(C^*_r(\pi))$, where  
\[
\varepsilon = 
\begin{cases}  
0 & \text{if $d $ is even} \\
1  & \text{if $d $ is odd.} 
\end{cases}
\]
\end{theorem}

\end{num}

\begin{num}
\label{number-eventual-defs}
As we already noted,  we shall actually need a slightly stronger version of Theorem~\ref{thm-bounded-homotopy-invariance}. It involves the following concept:

\begin{definition*}
	Let $Z$ and $W$ be locally compact Hausdorff spaces. An \emph{eventual map} from $Z$ to $W$ is  a continuous and proper map
	\[
	Z_1 \longrightarrow W ,
	\]
where $Z_1$ is   a closed subset of  $ Z$ whose complement has compact closure.  Two eventual maps from $Z$ to $W$ are \emph{equivalent} if they agree on the  complement of some compact set in $Z$.
\end{definition*}

There is an obvious category of locally compact spaces and equivalence classes of eventual maps. Let us denote the morphisms as follows: $ Z \morphism  W$.

\begin{definition*}
Two eventual morphisms 
\[
f_0, f_1 \colon Z \morphism  W
\]
 are \emph{homotopic} if there is an eventual morphism 
$g\colon Z\times [0,1]\morphism W$  for which the compositions 
\[
Z \rightrightarrows Z{\times}[0,1] \stackrel g \morphism W
\]
with the inclusions of $Z$ as $Z{\times}\{0\}$ and $Z{\times} \{ 1\}$ equal to  $f_0$ and $f_1$, respectively, in the category of locally compact spaces and eventual maps.
\end{definition*}

Homotopy is an equivalence relation on morphisms  that is compatible with composition, and so we obtain a notion of \emph{eventual homotopy equivalence}.  
\end{num}

\begin{num} It is clear how to define the concept of continuously controlled eventual morphisms between locally compact spaces that are continuously controlled over $\R$, as well as eventual  homotopy equivalences  that are continuously controlled over $\R$. And we can speak of  \emph{boundedly  controlled} eventual homotopy equivalences  between bounded geometry   Riemannian manifolds  $Z$ and $W$, and compatibility of these with $\pi$-principal bundles on $Z$ and $W$.
 \end{num}
 
 \begin{num}
   An orientation of a manifold $Z$ of dimension $d$ is a class in
    $H^{\text{lf}}_d(Z)$, the homology of locally finite chains, that maps to a generator in each group $H^{ {\text{lf}}}_d (Z, Z\setminus\{z\})$.   We shall call the image of the orientation class under the map 
 \[
 H^{\text{lf}}_d(Z) \longrightarrow \varinjlim_K H^{\text{lf}}_d (Z, K) ,
 \]
 where the direct limit is over the compact subsets of $Z$, the associated
 \emph{eventual orientation class}.  The direct limit is functorial for
 eventual morphisms, and so we can speak of an eventual  homotopy equivalence being \emph{orientation-preserving}.

 The  invariance property   that we shall use to prove Theorem~\ref{thm-main-theorem} is as follows:

\begin{theorem}
\label{thm-eventual-homotopy-invariance}
Let $Z$ and $W$ be  $(d{+}1)$-dimensional, connected complete
 {oriented}  Riemannian manifolds  with bounded geometry  that are continuously controlled over $\R$.  If  $\widetilde Z$ and $\widetilde W$ are principal $\pi$-bundles over $Z$ and $W$, and if there is a 
 boundedly controlled eventual homotopy equivalence between $Z$ and $W$ that
 is continuously controlled over $\R$,  orientation-preserving, and
 compatible with $\widetilde Z$ and $\widetilde W$, then 
\[
2^\varepsilon \big(\sig _{c,\pi}  (\widetilde Z/ Z) - \sig _{c,\pi}  (\widetilde W/ W) \big) = 0 
\]
in $K_d(C^*_r(\pi))$,  where  
\[
\varepsilon = 
\begin{cases}  
0 & \text{if $d $ is even} \\
1  & \text{if $d$ is odd.} 
\end{cases}
\]
\end{theorem}

We shall prove this result in Section~\ref{sec-eventual-htp-eq-invariance}.
 \end{num}

\begin{num}\label{product} Finally, we shall need a formula for the signature of a product manifold. Let $X$ be a smooth, closed and oriented manifold of dimension $p$, and $\widetilde X \to X$ a $\pi$-principal  bundle over $X$. Similarly,  let $Y$ be a smooth, closed and oriented manifold of dimension $q$, and $\widetilde Y \to Y$ a $\sigma$-principal  bundle over $Y$. There is a natural product map
\[
-\otimes-\colon K_p(C^*_r(\pi))\otimes K_q(C^*_r(\sigma)) \longrightarrow K_{p+q}(C^*_r(\pi {\times}\sigma))
\]
and   we have the following product formula: 
\[
\sig_{\pi{\times}\sigma}  (\widetilde X{\times}\widetilde{Y}/ X{\times }Y) =  2^{\varepsilon(X,Y)} \sig_{\pi}(\widetilde X/X)\otimes \sig_{\sigma}(\widetilde Y/Y)
\]
where 
\[
\varepsilon(X,Y) = 
\begin{cases}  
0 & \text{if one of $X$ or $Y$ is even-dimensional} \\
1  & \text{if both $X$ and $Y$ are   odd-dimensional.} 
\end{cases}
\]
This will be proved in Section~\ref{sec-products}.\end{num}

\section{Proof of the Codimension Two Theorem}
\label{sec-proof}

In this section, we shall prove  Theorem $\ref{thm-main-theorem}$ using the properties of the controlled signature that were listed above.

Let $h\colon N \to M$ be an orientation-preserving, smooth homotopy equivalence between smooth, closed, oriented manifolds of dimension $d+2$. Let $X$ be a smooth, closed, oriented submanifold of $M$ of codimension $2$ (and hence dimension $d$), and assume that 
the smooth map  
$ h: N \to M$
is transverse to $X$, so that 
the inverse image 
\[ Y = h^{-1}[X]  \]
is a smooth, closed, oriented, codimension-two submanifold of $N$. Assume in addition that 
\begin{enumerate}[\rm (i)]
	\item\label{it:pi1} $\pi_1(X) \to \pi_1(M)$ is injective;
	\item \label{it:pi2} $\pi_2(X)\to \pi_2(M)$ is surjective; and 
	\item the normal bundle of $X$ in $M$ is trivializable.
\end{enumerate}

 We fix a base point in $X$   and assume that $X$ is connected.
With assumption (i) above, we shall view $\pi_1(X)$  as a subgroup of $\pi_1(M)$. 	Let 
\[
p\colon \widehat  M\longrightarrow M
\]
be the covering of $M$ corresponding to the subgroup $\pi_1(X)\subset \pi_1(M)$. In other words, $\widehat M$ is the quotient space of the universal covering of $M$ by the subgroup $\pi_1(X)$. In particular, we have the following lemma.

\begin{lemma} 
 The inverse image of $X$ under  the   projection from 
$ \widehat M$ to $M$   is  a disjoint union of copies of coverings of $X$. The
component of the base point is homeomorphic to $X$, with homeomorphism given
by the restriction of the covering projection.\qed
\end{lemma}

Fix  the base point copy of  $X$ in the inverse image $p^{-1}(X)$.  We will  use the same notation, $X$ for this lift of the submanifold $X\subseteq M$:
  \[
X \subseteq p^{-1}(X) .
\]
 Denote by  $D(X)$   a closed tubular neighborhood of $X$ in $\widehat M$ and consider the smooth manifold
\[
\widehat  M\setminus  \mathring{D}(X) 
\]
(the circle denotes the interior) with boundary 
\[
\partial \bigl ( \widehat  M\setminus  \mathring{D}(X) \bigr ) \cong X\times S^1.
\]
The following lemma is contained in the proof of \cite[Theorem 4.3]{HankePapeSchick15}. 

\begin{lemma}
Under the assumptions \eqref{it:pi1} and \eqref{it:pi2} made above, the inclusion
\[
X\times S^1 \hookrightarrow  \widehat M  \setminus  \mathring{D}(X)
\]
induces  a split injection
\[
\pi_1(X\times S^1)   \to \pi_1(\widehat  M\setminus \mathring{D}(X)) 
\] 
on fundamental groups. \qed
\end{lemma} 	

Glue two copies of $\widehat M\setminus  \mathring{D}(X)$ along $X\times S^1$ so as to  form the space 
\[ 
	W =  \widehat  M\setminus  \mathring{D}(X)\cup_{X\times S^1} \widehat  M\setminus  \mathring{D}(X). 
\] 
By applying the  Seifert-van Kampen theorem, we obtain the following lemma, as
in the proof of \cite[Theorem 4.3]{HankePapeSchick15}. 
\begin{lemma} 
\label{lem-splitting2}
The map 
\[
\pi_1(X\times S^1) \longrightarrow  \pi_1(W)
\]
that is  induced from the inclusion $X\times S^1\hookrightarrow W$  is a split injection. \qed
 \end{lemma}

Pull back the covering 
\[
p: \widehat M \longrightarrow M
\] to $N$ along the map $h$, as indicated in the following diagram:
\[
\xymatrix{
\widehat N \ar[r]^{\widehat h} \ar[d] & \widehat M \ar[d]^p \\ 
N \ar[r]_h & M .
}
\]
We denote by
\[
Y = \widehat h^{-1}(X) \subset \widehat  N, 
\]
the inverse image under $\widehat h$ of our chosen    copy of $X$ in $\widehat  M$. 
\begin{lemma}
\label{lem-bdy-Y}
 The normal bundle of	$Y$ in $N$, or equivalently in  $\widehat  N$,  is  trivial. 
\end{lemma}
\begin{proof}
By transversality,  the normal bundle of $Y$ in $N$ is the pullback of
  the normal bundle of $X$ in $M$, and hence is trivial by  our assumption (iii).
\end{proof}

Now let us  prove Theorem~\ref{thm-main-theorem}, assuming Theorem~\ref{thm-eventual-homotopy-invariance} on the eventual homotopy invariance of the controlled signature (the rest of the paper will be devoted to proving Theorem~\ref{thm-eventual-homotopy-invariance}). 

 \begin{proof}[Proof of Theorem $\ref{thm-main-theorem}$]
Form the space 
\[ 
Z= \widehat  N\setminus \mathring{D}(Y)\,\, \cup_{Y\times S^1} \,\, \widehat  N\setminus  \mathring{D}(Y).   
\]
Here $\mathring{D}(Y)$ is the inverse image of $\mathring{D}(X)$ under $\widehat h$, which is an open disk bundle of $Y$ in $\widehat  N$, with boundary $ Y\times S^1 $ thanks to Lemma~\ref{lem-bdy-Y}.

The map $\widehat  h: \widehat  N \to \widehat M$  induces a canonical continuous map 
\begin{equation}
\label{eq-h-hat}
 \varphi\colon  Z \longrightarrow  W.
 \end{equation}
 
Although it was constructed using a process that started from a homotopy
equivalence, the map $\varphi$  in \eqref{eq-h-hat} is not a homotopy
equivalence in general. However, we still get an eventual homotopy
  equivalence.

\begin{lemma}
 The map  $\varphi$ is a boundedly controlled eventual homotopy equivalence
    that is continuously controlled over $\mathbb{R}$, with respect to the
    control  maps
  \[ 
  c_W\colon W \longrightarrow  \mathbb R \quad\text{and}\quad   c_Z\colon Z\longrightarrow 
      \mathbb{R} ,
      \]
  where
    \[  {c_W}(w) = \begin{cases}
        d(w, X\times S^1) & \textup{ if $w$ is in the first copy $M\setminus  \mathring{D}(X)$ in $W$, } \\
        -d(w, X\times S^1) & \textup{ if $w$ is in the second copy
          $M\setminus \mathring{D}(X)$ in $W$, }
      \end{cases}\]  
      and where $c_Z$ is constructed correspondingly. Here
    $d(w, X\times S^1)$ is the distance between $w$ and $ X\times S^1$ in $W$.
  \end{lemma}
  
 \begin{proof}
 The map $h$, its homotopy inverse $g\colon M\to N$, and the homotopies
  $H_1\colon N\times [0,1]\to M$ and  $H_2\colon M\times [0,1]\to N$ between the
  compositions and the identity maps all lift to maps $\hat g,\hat h,\hat
  H_1,\hat H_2$ on $\hat M$, $\hat N$. In general, only $\hat h$ will map
  $\mathring{D}(Y)\subset \hat N$  to $\mathring{D}(X)\subset \hat M$ and
  therefore give rise to an honest map $\phi\colon Z\to W$. But the other maps
  at least give rise to eventual maps which establish the required eventual
  homotopy equivalence.
 
 We lift the Riemannian metrics from the compact manifolds $N,M$ to $\hat N,\hat
  M$ and obtain metrics on $W$ and $Z$ by smoothing these metrics in a
  compact neighborhood of $X\times S^1$ or $Y\times S^1$,
  respectively. Therefore, we obtain metrics of bounded geometry. All our maps
  are lifted from the compact manifolds $M,N$ and therefore are boundedly
  controlled. By construction, they also preserve the decomposition into
  positive and negative parts according to the control maps $c_W,c_Z$ (up to a
  compact deviation).  That is, they are continuously controlled over $\mathbb{R}$.
  \end{proof}

Let $r\colon \pi_1(W) \to \pi_1(X\times S_1)$ be the splitting of the inclusion homomorphism $\iota \colon \pi_1(X\times S^1) \to \pi_1 (W)$. Let $\widetilde W$ be the covering space of $W$ corresponding to the subgroup $\ker(r)$, that is,  the quotient space of the universal cover of $W$ by $\ker(r)$.   Let $\widetilde Z$ be the covering space of $Z$ that is the pullback of $\widetilde W \to W$ along the map $\varphi$.    It follows from  Theorem $\ref{thm-eventual-homotopy-invariance}$ that 
	\[ 2^\varepsilon\left( \sig_{c,\pi}(\widetilde W/W) - \sig_{c,\pi }(\widetilde Z/Z) \right) = 0 \quad  \text{in $K_{d}(C^*_r(\pi))$.}  \]
	where $\pi = \pi_1(X\times S^1) =  \pi_1(X) \times \mathbb Z$ and
	\[
	\varepsilon = 
	\begin{cases}  
	0 & \text{if $d $ is even} \\
	1  & \text{if $d$ is odd.} 
	\end{cases}
	\] 
Equivalently, we have 
\[ 2^\varepsilon \left( \sig_\pi(\widetilde{X\times S^1}/X\times S^1)  - \sig_\pi (\widetilde{Y\times S^1}/Y\times S^1) \right) = 0 \quad  \text{in $K_{d}(C^*_r(\pi))$.}  \]
Here $\widetilde{X\times S^1}$ is the universal cover of $X\times S^1$, which
pulls back to give a covering space  $\widetilde{Y\times S^1}$ over $Y\times
S^1$. It is a basic computation of higher signatures that
  $\sig_{\mathbb{Z}}(S^1)$ is a generator of $K_1(C^*_r(\mathbb{Z}))\cong
  \mathbb{Z}$. Moreover, we have for an arbitrary group $\Gamma$ the K\"unneth
  isomorphism given by the product map
  \begin{equation*}
    K_p(C^*_r(\Gamma))\otimes K_0(C^*_r(\mathbb{Z})) \oplus
    K_{p-1}(C^*_r(\Gamma))\otimes K_1(C^*_r(\mathbb{Z})) \xrightarrow{\cong}
    K_p(C^*_r(\Gamma\times \mathbb{Z})).
  \end{equation*}
The theorem now follows from the product formula of signature operators (cf. Section $\ref{product}$).   
\end{proof}

We conclude this section with a discussion of    counterexamples to some would-be strengthenings of the theorem, as well as  some further comments.

Techniques from surgery theory, and in particular Wall's $\pi$-$\pi$ theorem  \cite[Theorem 3.3]{Wall99},  show that conditions (i) and (ii) in Theorem \ref{thm-main-theorem} are   both necessary.  
 Concerning condition (i), let us start by noting that the $n$-torus $T^n$
 embeds into the $(n{+}2)$-sphere $S^{n+2}$    with trivial normal bundle
 (indeed it maps into every $(n{+}2)$-dimensional manifold with trivial normal
 bundle, but one embedding into the sphere will suffice).  Let  $f: V\to W$ be
 a  degree-one normal map between simply connected, closed, oriented
 $4k$-manifolds with distinct signatures $\sig(V)\ne \sig(W)\in\mathbb{Z}$.
 Consider the degree-one normal map of pairs
 \[
{f{\times} \mathrm{Id}} \colon  \bigl  (V {\times} B^{n+3}, V{\times} S^{n+2}\bigr )    \longrightarrow 
 \bigl  (W {\times} B^{n+3}, W{\times} S^{n+2} \bigr ) .
 \]
 According to Wall's $\pi$-$\pi$ theorem       this is normally cobordant to a homotopy equivalence of pairs; let us write the homotopy equivalence on boundaries as 
 \[
 h   \colon   N    \xrightarrow{\simeq}    W{\times} S^{n+2} .
 \]
 Consider the codimension-two submanifold (with trivial normal bundle)
\[
 W {\times} T^n \subseteq W{\times} S^{n+2}   .
 \]
Its signature   class in the free abelian group $K(C^*_r (\mathbb Z^n))$ is
$\sig(W)\cdot \sig_{\mathbb{Z}^n}(T^n)$ where $\sig_{\mathbb{Z}^n}(T^n)$ has
infinite order.  However the  signature class of the
transverse inverse image $M=h^{-1}(W\times T^n)$  is equal to
\[\sig_{\mathbb Z^n}(M)=\sig(V)\cdot\sig_{\mathbb{Z}^n}(T^n)\in
  K(C^*_r(\mathbb Z^n)).\]
This follows from 
the cobordism invariance of the signature class, using that the normal bordism
between $f\times\mathrm{id}_{S^{n+2}}$ and $h$ restricts to a normal bordism
between $f\times\mathrm{id}_{T^n}$ and $h|_M$.  Therefore (even
twice) the signature classes of the two submanifolds of codimension $2$ are
distinct.

As for  condition (ii),  start with  any   degree-one normal map $f\colon
      W \to X$ such that   
      \[
      2 \left(\sig_{\pi_1(X)}(\widetilde W / W) -  \sig_{\pi_1(X)}(\widetilde X/ X) \right) \ne  0 
      \]
       (here $\widetilde X$ is the universal cover of $X$, and $\widetilde W$
       is the pullback to $W$ along $f$). By Wall's theorem again, the map 
    \[
    f\times \id \colon W{\times} S^2\longrightarrow  X{\times} S^2
    \]
     is   normally cobordant to an orientation-preserving  homotopy
     equivalence
     \[ 
     h \colon N \stackrel \simeq\longrightarrow  X{\times} S^2.
     \]
    Let $Y$ be the transverse inverse image of the codimension-two submanifold $X{\times}\{\mathrm{pt}\}\subseteq X{\times} S^2$ under the map $h$. It follows from cobordism invariance of the signature class that 
\[
\sig_{\pi_1(X)}(\widetilde Y/Y )=\sig_{\pi_1(X)}(\widetilde W/W) ,
\]
and therefore that 
      \[ 2\left(\sig_{\pi_1(X)}(\widetilde Y/Y)- \sig_{\pi_1(X)}(\widetilde
        X/X)\right)\ne 0,  \] 
as required.

      On the other hand, it seems to be difficult to determine whether or not 
      condition (iii)   is necessary. This   question is also
      open for the companion result about positive scalar curvature; compare
      in particular \cite{MGBL83}.

Finally, as we mentioned in the introduction, we do not know if the factor of $2$ that appears in Theorem~\ref{thm-main-theorem} is really necessary, although we suspect it is.  
      Let $p\colon M \to  \Sigma$ be any bundle of oriented closed manifolds over an
      oriented closed surface $\Sigma\ne S^2$ with fiber $X = p^{-1}(\text{\rm pt})$. If $h\colon N\to M$ is a homotopy 
      equivalence, then, whether or not the composition 
      \[
 N\stackrel h \longrightarrow M \stackrel p\longrightarrow 
      \Sigma
      \]
       is homotopic to a fiber bundle projection, and whether or not   
      the induced map from the ``fiber'' 
      $Y=(p\circ h)^{-1}(\text{\rm pt})$ to $X$
      is cobordant to a homotopy equivalence, Theorem \ref{thm-main-theorem} guarantees that the
      $C^*$-algebraic higher signatures of $Y$ and $X$ coincide in
      $K_*(C^*_r(\pi_1(X)))$, after multiplication with $2$.  In the related $L$-theory context, examples can be constructed where the factor of $2$ is really necessary, and perhaps the same is true here.  Moreover it seems possible that many nontrivial examples of the theorem can be obtained from this construction.

\section{Signature Classes in C*-Algebra K-Theory}
\label{sec-controlled-signature}

In this section we shall review the construction of the  $C^*$-algebraic signature class. 
One route towards the definition of the signature class goes via index  theory and the signature operator.  See for example \cite[Section 7]{BCH94} for an introduction.   Although we shall make use of the index theory approach, we shall mostly  follow a different route, adapted   from \cite{HigsonRoeSurgToAnal1,HigsonRoeSurgToAnal2}, which makes it easier to handle the invariance properties of the signature class that we need.

\subsection{Hilbert-Poincar\'{e} complexes}\label{sec:hp}
Hilbert-Poincar\'{e} complexes were introduced in  \cite{HigsonRoeSurgToAnal1} to adapt the standard symmetric signature constructions in $L$-theory to the context of $C^*$-algebra $K$-theory.  We shall review the main definitions here; see  \cite{HigsonRoeSurgToAnal1}  for more details.

\begin{definition}\label{def:hilpoin}
Let $A$ be a unital $C^\ast$-algebra. 
	An \emph{$d$-dimensional Hilbert-Poincar\'{e} complex}  over $A$  is a complex of finitely generated (and therefore also projective) Hilbert $A$-modules and bounded, adjointable differentials,
	\[  E_0 \xleftarrow{b} E_{1} \xleftarrow{b} \cdots \xleftarrow{b} E_{d}, \]
	together with bounded adjointable operators $D: E_p \to E_{d-p}$ such that 
	\begin{enumerate}[(i)]
		\item if $v\in E_{p}$, then $D^\ast v = (-1)^{(d-p)p} Dv$;
		\item if $v\in E_{p}$, then $Db^\ast v + (-1)^{p} bDv = 0$; and
		\item $D$ is a homology isomorphism  from  the dual complex 
		\[  E_d \xleftarrow{b^\ast} E_{d-1} \xleftarrow{b^\ast} \cdots \xleftarrow{b^\ast} E_{0} \]
		to the complex $(E, b)$. 
	\end{enumerate}
\end{definition}

To each  $d$-dimensional Hilbert-Poincar\'{e} complex over $A$ there is associated a  \emph{signature} in   $C^*$-algebra $K$-theory:
\[
\sig (E, D)\in K_d(A).
\]
See \cite[Section 3]{HigsonRoeSurgToAnal1} for the construction. If $A$ is the $C^*$-algebra of complex numbers, and $d\equiv 0 \mod 4$, then the signature identifies with the signature of the  quadratic form induced from $D$ on middle-dimensional homology.

 The signature is a homotopy invariant   \cite[Section 4]{HigsonRoeSurgToAnal1} and indeed a bordism invariant  \cite[Section 7]{HigsonRoeSurgToAnal1}.   These concepts will be illustrated in the following   subsections.

\subsection{The  Signature Class }
\label{sec-signature-def}

Let $\Sigma$ be a finite-dimensional simplicial complex and denote by $\C [
\Sigma_p] $ the vector space of finitely-supported, complex-valued functions
on the set of $p$-simplices in $\Sigma$.  After assigning orientations to
simplices  we obtain in the usual way a  {chain} complex
\begin{equation}
\label{eq-homology-complex}
\C[\Sigma_0] \xleftarrow{b_1} \C [ \Sigma_1] \xleftarrow{b_{2}} \cdots .
\end{equation}
If $\Sigma$ is, in addition, locally finite (meaning that each vertex is contained in only finitely many simplices) then there is also an adjoint complex
\begin{equation}
\label{eq-cohomology-complex}
\C[\Sigma_0] \xrightarrow{b_1^*} \C [ \Sigma_1] \xrightarrow{b_{2}^*} \cdots ,
\end{equation}
in which the  differentials are adjoint to those of \eqref{eq-homology-complex} with respect to the standard inner product for which the delta functions on simplices are orthonormal.

If $\Sigma$ is furthermore an oriented combinatorial manifold of dimension $d$ with fundamental cycle $C$, then there is a duality operator 
\begin{equation}
\label{eq-duality-operation}
D: \C[\Sigma_p]  \longrightarrow \C[\Sigma_{d-p}]
\end{equation}
defined by the usual formula
$D(\xi) = \xi\cap C $
(compare \cite[Section 3]{HigsonRoeSurgToAnal2}, where the duality operator is called $\mathbb P$). There is also an adjoint operator 
\begin{equation*}
D^*: \C[\Sigma_{d-p}]  \longrightarrow \C[\Sigma_{p}] .
\end{equation*}
We are, as yet, in a purely algebraic, and not $C^*$-algebraic, context,  but the operator $D$ satisfies all the   relations   given in Definition~\ref{def:hilpoin}, except for the first.  As  for the  first, the operators $D$ and $(-1)^{(d-p)p} D^*$ are chain homotopic, and   if we replace $D$ by the average
\[
\tfrac 12 ( D+ (-1)^{(d-p)p} D^*),
\]
 which we shall do from now on, then  all the relations are satisfied.

 In order to pass from complexes of vector spaces to complexes of Hilbert modules we shall assume the following:

\begin{definition}\label{def:uniformbd}
	A simplicial complex $\Sigma$ is of  \emph{bounded geometry} if there is a positive integer $k$ such that any   vertex of $\Sigma$ lies in at most $k$ different simplices of $\Sigma$.
\end{definition}

\begin{definition}
\label{def-proper-control}
Let $P$ be a proper metric space and let $S $ be   a set. A function 
\[
 c  \colon S \longrightarrow    P
 \]
is a \emph{proper control map} if for every $r>0$, there is a bound $N< \infty$ such that if $B\subseteq P$ is any subset of diameter less than $r$, then $c^{-1}[B]$  has  cardinality less than $N$. 
\end{definition}

 \begin{definition}
 \label{def-C-of-Z-and-W}
 Let $P$ be a proper metric space, let $S$ and $T$ be   sets, and let 
\[
c_S \colon S  \longrightarrow    P
\quad \text{and} \quad 
 c_T \colon T \longrightarrow    P
 \]
be proper control maps.  Suppose in addition that a discrete group $\pi$ acts on  $P$ properly through isometries, and on  the sets $S$ and $T$, and suppose that   the control maps  are equivariant. 
	A linear map 
\[
	A\colon \C[S] \longrightarrow  \C[T]
\] 
is \emph{boundedly geometrically controlled} over $P$ if 
	\begin{enumerate}[(i)]
		\item the matrix coefficients of $A $ with respect to the  bases  given by the points of $S$ and $T$ are uniformly bounded; and 
		\item   there is a constant $K>0$ such that the  $(t,s)$-matrix coefficient of $A$ is zero whenever $d(c(t), c(s))>K$. 
			\end{enumerate}
We shall denote by 
$
\C_{P,\pi} [T,S]
$
 the linear space   of  equivariant, boundedly geometrically controlled linear operators from $\C[S]$ to $ \C[T]$.	 \end{definition}

Every  boundedly geometrically controlled linear operator is adjointable, and the  space  $\C_{P,\pi}[S,S]$ is a $*$-algebra by composition and adjoint of operators.

\begin{definition} 
\label{def-red-completion}
We shall denote by $C^* _{P,\pi}(S,S)$   the $C^*$-algebra completion of $\C_{P,\pi}[S,S]$ in the operator norm on $\ell ^2 (S)$.
\end{definition}

\begin{remark} 
\label{rem-max-completion}
The norm above is related to the norm in   \emph{reduced} group $C^*$-algebras.  There is also a norm on $\C_{P,\pi}[S,S]$  appropriate to full group $C^*$-algebras, namely
\[
\| A\| = \sup _\rho \| \rho (A) \colon H \to H\| ,
\]
where the supremum is over all  representations of the $*$-algebra $\C_{P,\pi} [S,S ]$ as bounded operators on Hilbert space.  {See} \cite{GongWangYu08}, where among many other things it is shown that the supremum above is finite. 
\end{remark}

Now the space  $\C_{P,\pi} [T,S]$ in Definition~\ref{def-C-of-Z-and-W} is a right    module over $\C_{P,\pi}[S,S]$ by composition of operators.

\begin{lemma}
\label{lemma-f-g-projective}
If there is an equivariant, injective  map $T \to S$ making the diagram 
\[
\xymatrix{
T  \ar[r]\ar[d]_{c_T} & S \ar[d]^{c_S} \\
P \ar@{=}[r] &P
}
\]
commute, 
then $\C_{P,\pi} [T,S]$ is finitely generated and projective over ${\C}_{P,\pi}[S,S]$. \qed
\end{lemma}

\begin{proof}
  Define $p\colon \C[S]\to \C[S]$ sending $t\in T$ to $t$ and $s\in S\setminus
  T$ to $0$. Then $p\in \C_{P,\pi}[S,S]$ is an idempotent with image precisely
  $\C_{P,\pi}[T,S]$. So the image is projective and generated by the 
  element $p$.
\end{proof}

Assume that the hypothesis in the lemma holds (we shall always assume this in what follows).  The induced module 
\[
C^*_{P,\pi}(T,S) = \C_{P,\pi} [T,S] \otimes _{\C_{P,\pi}[S,S]} C^*_{P,\pi}(S,S)
\]
is then a finitely generated and projective module over $C^*_{P,\pi}(S,S)$.  It is isomorphic to the completion of $ \C_{P,\pi} [T,S] $ in the norm associated to the   $\C_{P,\pi}[S,S]$-valued, and hence $C^*_{P,\pi}(S,S)$-valued,   inner product  on $\C_{P,\pi}[ {T,S}]$ defined by 
\[
\langle A,B\rangle = A^*B .
\]
 So it is a finitely generated and projective Hilbert module over $C^*_{P,\pi}(S,S)$.

Let us   return to the bounded geometry, combinatorial manifold $\Sigma$, which we shall now assume is equipped with an action of $\pi$.  Assume that the sets $\Sigma_p$ of $p$-simplices admit proper control maps to some $P$ so that there is a uniform bound on the distance between the image of any simplex and the image of any of its vertices.   Then the differentials in \eqref{eq-homology-complex} and the duality operator \eqref{eq-duality-operation} are boundedly geometrically controlled, and   we obtain from them a complex
\begin{equation}
   C^*_{P,\pi}(  \Sigma_0,S) \xleftarrow{b_1}  C^*_{P,\pi}( \Sigma_1,S) \xleftarrow{b_{2}} \cdots \xleftarrow{b_{d}}  C^*_{P,\pi}(\Sigma_d,S),
\end{equation}
of finitely generated and projective Hilbert modules over $C^*_{P,\pi}(S,S)$ for any $S$ that satisfies the hypothesis of Lemma~\ref{lemma-f-g-projective} for all $T=\Sigma_p$. 
The procedure in Section~\ref{sec:hp} therefore gives us a signature
\begin{equation}
\label{eq-sig-in-pi-S-algebra}
\sig_{P,\pi}( \Sigma,S)\in K_d (C^*_{P,\pi}(S,S)).
\end{equation}

If  $S$ is a free $\pi$-set with  {finitely many $\pi$-orbits}, and if we regard the underlying set  of $\pi$ as a   $\pi$-space by  left multiplication, then the bimodule $\C_{P,\pi}[\pi ,S]$ is a Morita equivalence 
\[
\C_{P,\pi}[S,S]\underset{\text{Morita}}\sim \C_{P,\pi}[\pi ,\pi] ,
\]
and upon completion we obtain a $C^*$-algebra Morita equivalence 
\begin{equation}
\label{eq-Morita}
C_{P,\pi}^*(S,S)\underset{\text{Morita}}\sim C^*_{P,\pi}(\pi,\pi) .
\end{equation}
But the action of $\pi$ on the vector space  $\C[\pi]$   by \emph{right} translations  gives an isomorphism of $*$-algebras
\[
\C[\pi] \stackrel \cong \longrightarrow \C_{P,\pi}[\pi,\pi]
\]
 and then by completion an isomorphism 
\begin{equation}
\label{eq-C-star-pi-iso}
 C^*_r(\pi ) \stackrel \cong \longrightarrow C^*_{P,\pi}(\pi,\pi).
 \end{equation}
Putting \eqref{eq-Morita} and \eqref{eq-C-star-pi-iso} together we obtain a canonical isomorphism 
\begin{equation}
\label{eq-K-theory-iso}
 K_*(C^*_{P,\pi}(S,S)) \cong K_*(C^*_r(\pi)).
 \end{equation}

 Now let $X$ be a smooth, closed and oriented manifold of dimension $d$, and let 
$
\widetilde X$ be a $\pi$-principal bundle over $X$. Fix any triangulation of  $X$ and lift it to a triangulation $\Sigma$ of $\widetilde X$.  In addition, fix a Riemannian metric on $X$ and lift it to a Riemannian metric on $\widetilde X$.  Let $P$ be the underlying proper metric space.   If we choose $S$  to be the the disjoint union of all $\Sigma_p$, then
we obtain a signature as in \eqref{eq-sig-in-pi-S-algebra} above.

\begin{definition}
We denote by
\[
\sig_\pi(\widetilde X/X)\in K_d (C^*_r (\pi))
\]
the signature class associated to \eqref{eq-sig-in-pi-S-algebra} under the $K$-theory isomorphism \eqref{eq-K-theory-iso}.
\end{definition}

\begin{remark}
\label{rem-max-completion3}
By using the $C^*$-algebra norm in Remark~\ref{rem-max-completion} we would
instead obtain a    signature class in $K_d (C^*(\pi))$, where $C^*(\pi)$ is the maximal group $C^*$-algebra.
 \end{remark}
 
 The general invariance properties of the signature of a Hilbert-Poincar\'e complex imply that the signature is an oriented homotopy invariant, as in \eqref{num-homotopy-invariant} and a bordism invariant, as in \eqref{num-bordism-invariant}.  Compare \cite[Sections 3\&4]{HigsonRoeSurgToAnal2}.
 
 \subsection{Signature Classes in Coarse Geometry}
 
 Let $P$ be a proper metric space. A \emph{standard $P$-module} is a separable Hilbert space that is equipped with a nondegenerate representation of the $C ^*$-alg\-ebra $C_0(P)$.  It is unique up to finite-propagation unitary isomorphism. As a result, the $C^*$-algebra $C^*(P)$ generated by finite-propagation, locally compact operators is unique up to inner (in the multiplier $C^*$-algebra) isomorphism, and its topological $K$-theory is therefore unique up to canonical isomorphism.  For all this see for example  \cite[Section 4]{HigsonRoeYu93}. 
 
  If a discrete group $\pi$ acts properly and isometrically on $P$, then all these  definitions and constructions can be made equivariantly, and we shall denote by  $C^*_{\pi}(P)$  the $C^*$-algebra generated by the $\pi$-equivariant, finite-prop\-agation, locally compact operators. 
 
 If $S$ is a $\pi$-set that is equipped with an equivariant  proper control map to $P$, as in Definition~\ref{def-proper-control}, then $\ell^2(S)$ carries a $\pi$-covariant representation of $C_0(P)$ via the control map, and there is an isometric, equivariant,  finite-propagation inclusion of $\ell^2(S)$ into any $\pi$-covariant  standard $P$-module; see  \cite[Section 4]{HigsonRoeYu93} again.  This induces an inclusion of $C^*_{P,\pi}(S,S)$ into $C^*_{\pi}(P)$, and a canonical homomorphism
 \begin{equation}
 \label{eq-map-to-roe-algebra}
 K_*(C^*_{P,\pi}(S,S)) \longrightarrow  K_*( C^*_{\pi}(P)).
\end{equation}

\begin{remark}
\label{rem-P-mod-pi-cpt}
If the quotient space $P/\pi$ is compact, then this is an \emph{isomorphism}, and, as we noted, in this case the left-hand side of \eqref{eq-map-to-roe-algebra} is canonically  isomorphic to $K_*(C^*_r(\pi))$.   But in general the right-hand side can be quite different from either $ K_*(C^*_{P,\pi}(S,S))$ or $K_*(C^*_r(\pi))$.
\end{remark}

 \begin{definition}
 \label{def-coarse-signature}
Let $W$ be a complete, bounded geometry, oriented Riemannian manifold, and let $\widetilde W$ be a $\pi$-principal bundle over $W$.  We shall denote by 
 \[
 \sig_{\pi}(\widetilde W/W) \in K_*( C^*_{\pi}(\widetilde W))
  \]
  the signature class that is obtained from a $\pi$-invariant, bounded geometry triangulation
  $\Sigma$ of $\widetilde W$ by applying the map \eqref{eq-map-to-roe-algebra} to the class    \eqref{eq-sig-in-pi-S-algebra}.
 \end{definition}
 
 This is the same signature class  {as} the one considered in
 \cite{HigsonRoeSurgToAnal2}.  Once again, the general invariance properties
 of the signature of a Hilbert-Poincar\'e complex  immediately provide a
 homotopy invariance result:
 
\begin{theorem} 
The signature class $  \sig_{\pi}(\widetilde W/W)$ is independent of the triangulation.  In  {fact} the signature class is a boundedly controlled oriented homotopy invariant.
\end{theorem}

\subsection{The Signature Operator}
\label{sec-sig-op}

\begin{definition}
Let  $P$ be an oriented, complete   Riemannian manifold. We denote by $D_P$ the signature operator on $P$; see for example \cite[Section 5]{HigsonRoeSurgToAnal2}.  
\end{definition}

 If $W$ is an   oriented, complete   Riemannian manifold, and if $\widetilde W$ is a $\pi$-principal bundle over $W$, then  the operator $D_{\widetilde W}$  is a   $\pi$-equivariant, first-order elliptic differential operator on $\widetilde W$, and as explained in \cite{BCH94,Roe93,Roe96} it has an equivariant analytic index 
 \[
 \operatorname{Ind}_\pi (D_{\widetilde W}) \in K_d ( C^*_\pi (\widetilde W)).
 \]
  
 \begin{theorem}[{\cite[Theorems 5.5 and 5.11]{HigsonRoeSurgToAnal2}}]
 If $W$ is  a complete  bounded geometry Riemannian manifold of dimension $d$, and if $\widetilde W$ is a $\pi$-principal bundle over $W$, then 
 \[
  \sig_{\pi}(\widetilde W/W)  =  \operatorname{Ind}_\pi (D_{\widetilde W}) \in K_d ( C^*_\pi (\widetilde W)) .
 \]
 \end{theorem}

\subsection{Signature of a Product Manifold}
\label{sec-products}

The signature operator point of view on the signature class makes it easy to verify the product formula in Paragraph~\ref{product}. 

If one of $X$ or $Y$ is even-dimensional, then the signature operator on
$X\times Y$ is the product of the signature operators on $X$ and $Y$; if both
$X$ and $Y$ are   odd-dimensional, then the signature operator on $X\times Y$
is  {the direct sum of two copies of} the product of the signature operators on $X$ and $Y$. The product formula for the signature   class follows from the following commutative diagram:
\[  
\xymatrix{ K_p(B\pi ) \otimes K_q(B\sigma ) \ar[r] \ar[d]  &  K_{p+q}(B \pi {\times} B\sigma) \ar[d] \\
	K_p(C^\ast_r(\pi)) \otimes K_q(C^\ast_r(\sigma)) \ar[r] & K_{p+q}(C^*_r(\pi {\times}\sigma)) }
\]
where the horizontal arrows are the   products in $K$-homology and   $K$-theory, and the vertical maps are index maps (or in other words assembly maps); compare \cite{Chabert00}.
The argument uses the fact that the $K$-homology class of the product of two elliptic operators is equal to the   product of their corresponding $K$-homology classes.  
 
\subsection{{Signature of the Circle}}

The signature operator of $S^1$ is the operator $i \frac{d}{d\theta}$,   which is well known to
represent a generator of $K_1(B\mathbb{Z})\cong \mathbb{Z}$, and whose index
in $K_1(C^*_r(\mathbb{Z}))$ is well known to be a generator of
$K_1(C^*_r(\mathbb{Z}))\cong\mathbb{Z}$.

\subsection{Partitioned Manifold Index Theorem}

In this section we shall describe  a minor extension of Roe's partitioned manifold index theorem \cite{Roe88,Higson91}.  Compare also \cite{Zadeh10}.  We shall consider only the signature operator, but the argument applies to any Dirac-type operator.

Let $A$ be a $C^*$-algebra. If $I$ and $J$ are closed, two-sided ideals in $A$ with  $I+J = A$, then there is a   Mayer-Vietoris sequence in $K$-theory:
 \[   
 \xymatrix{
 	K_0(I\cap J) \ar[r] &  K_0(I) \oplus K_0(J) \ar[r] &  K_0(A)  \ar[d]^{\partial}  \\
 	      K_1(A) \ar[u]^{\partial}  & K_1(I) \oplus K_1(J)\ar[l]  &    K_1(I\cap J) . \ar[l] 
 } 
 \]
 See for example \cite[Section 3]{HigsonRoeYu93}. 
 The boundary maps may be described as follows.  There is an isomorphism
 \[
 I/ (I\cap J) \stackrel \cong \longrightarrow A /J ,
 \]
and the boundary maps are  the compositions
\[
K_*(A) \longrightarrow K_*(A/J) \cong K_*(I/(I\cap J)) \longrightarrow K_{*-1} (I\cap J),
\]
where the last arrow is the boundary map in the $K$-theory long exact sequence for the ideal $I\cap J \subseteq I$.  The same map is obtained, up to sign, if $I$ and $J$ are switched.

Now let $W$ be  a complete Riemannian manifold that    is continuously controlled over $\R$, and let $\widetilde W$ be a principal $\pi$-bundle	 over $W$.  Let $X\subseteq W$ be the transverse inverse image of $0\in \R$ and let $\widetilde X$ be the restriction of $\widetilde W$ to $X$.

Denote by $I$ the ideal in $A= {C}^*_\pi (\widetilde W)$ generated by all finite propagation, locally compact operators  that are supported in  $c^{-1}[a,\infty)$ for some $a\in \R$,  and  denote by $J$ the ideal  in $ A$ generated by all    finite propagation, locally compact operators that are supported on  $c^{-1}(-\infty, a]$ for some $a \in \R$. 
The sum $I +J$ is equal to $ A$.

The intersection $ 
I \cap J $  is generated by all operators that are supported in the inverse images under the control map of   compact subsets of $\R$.   
Any finite propagation isometry from a standard $\widetilde X$ module into a
standard  $\widetilde W$-module induces an inclusion of $C^*_\pi(\widetilde
X)$ into the intersection, and this inclusion induces a canonical isomorphism
in $K$-theory; see \cite[Section 5]{HigsonRoeYu93}. Following
Remark~\ref{rem-P-mod-pi-cpt}, since $X$ is compact we therefore obtain
canonical isomorphisms
 \[
 K_*(C^*_r(\pi)) \stackrel \cong \longrightarrow K_*(C^*_\pi (\widetilde X))\stackrel \cong \longrightarrow K_*(I \cap J).
 \]

 \begin{theorem}
 Let $W$ be  an oriented,   complete Riemannian manifold of dimension $d{+}1$ that is continuously controlled over $\R$, and let $\widetilde W$ be a $\pi$-principal bundle over $W$.
The Mayer-Vietoris boundary homomorphism
\[
\partial \colon K_{d{+}1} (C^*_\pi (\widetilde W)) \longrightarrow K_{d} (C^*_r (\pi))
\]
 maps the index class  of  {the signature operator} $D_{\widetilde W}$
 to the index class of  {the signature operator} $D_{\widetilde X}$ if
 $d$ is even, and to twice the index class of $D_{\widetilde X}$  if $d$ is
 odd.
\end{theorem}

\begin{proof} 
The argument of \cite{Roe88} reduces the theorem  to the case of a product manifold, that is, $W = X{\times} \mathbb R$ with the control map $c\colon X{\times} \mathbb R \to \mathbb R$ given by the projection to $\mathbb R$ (see also \cite{Higson91}). 
The case of the product manifold  $X{\times}\R$ follows from  a direct
computation, as in \cite{Higson91}.  Once again, the reason for the factor of
$2$, in the case when $d$ is odd, is that signature operator on $X{\times} \R$
is \emph{{the direct sum of two copies of}} the product of the signature operators of $X$ and $\R$ in this case.
\end{proof}
 
 \begin{corollary}\label{cor-partition}
 Assume that $W$ has bounded geometry.  The Mayer-Vietoris boundary homomorphism
\[
\partial \colon K_{d+1} (C^*_\pi (\widetilde W)) \longrightarrow K_{d} (C^*_r (\pi))
\]
 maps $\sig_\pi (\widetilde W/ W)$ to the transverse signature class $\sig_{c,\pi} (\widetilde W/ W)$, if $d$ is even, and it maps it to twice the transverse signature class of $W$, if $d$ is odd. \qed
 \end{corollary}

\subsection{The Transverse Signature}
\label{sec-transverse-signature}

We can now prove Theorem~\ref{thm-bounded-homotopy-invariance}.  Let $Z$ and $W$ be as in the statement.  The boundedly controlled homotopy equivalence between them identifies  $\sig_\pi(\widetilde W/ W)$ with  $\sig_\pi(\widetilde Z/ Z)$.  Since  the equivalence is continuously controlled over $\R$ there is a commuting diagram
\[
\xymatrix{
K_{d+1} (C^*_\pi (\widetilde W))\ar[d] \ar[r] &  K_{d} (C^*_r (\pi)) 
\ar@{=}[d] \\
K_{d+1} (C^*_\pi (\widetilde Z)) \ar[r]  &  K_{d} (C^*_r(\pi))   }
\]
involving Mayer-Vietoris boundary homomorphisms. Thanks to Roe's partitioned
manifold  {index theorem, in the form of Corollary~\ref{cor-partition},} the horizontal maps send  the signature classes $\sig_\pi(\widetilde W/ W)$ and $\sig_\pi(\widetilde Z/ Z)$ to $2^\varepsilon$ times the transverse signature  classes $  \sig_{c,\pi}(\widetilde W/ W)$  and $  \sig_{c,\pi}(\widetilde Z / Z)$, respectively.  So the two transverse signature classes, times $2^\varepsilon$, are equal.

\section{Invariance Under Eventual Homotopy Equivalences}
\label{sec-eventual-htp-eq-invariance}
 
 It remains to prove Theorem~\ref{thm-eventual-homotopy-invariance}. This is what we shall do  here, and it is at this point that we shall make proper use of the complex $C^*_\pi( \Sigma_{\bullet}, S)$ that we introduced in Section~\ref{sec-signature-def}. 

\subsection{Eventual Signature}

Let $P$ be proper metric space and assume that  a discrete group $\pi$ acts on $P$ properly and through isometries. 

\begin{definition}
Let $S,T$ be   sets equipped with $\pi$-actions, and let 
\[
c_S \colon S  \longrightarrow    P
\quad \text{and} \quad 
 c_T \colon T \longrightarrow    P
 \]
be $\pi$-equivariant proper control maps.  A boundedly geometrically controlled linear operator $A\colon \C [ S ]\to  \C [ T ]$ is  \emph{$\pi$-compactly supported} if   there exists a  closed subset $F \subseteq P$ whose quotient by $\pi$ is compact such that $A_{t,s} =0$  unless $c_S(s) \in F$ and $c_T(t) \in F$.
Define 
 \[
 C_{P,\pi,\cpt}^\ast(T,S) \subseteq C^*_{P,\pi} (T,S)  
 \]
 to be the operator norm-closure  of the linear space of   $\pi$-compactly supported, boundedly geometrically controlled  operators. 
\end{definition}

When $S=T$ the space $ C_{P,\pi,\cpt}^\ast(S,S) $ is a closed two-sided ideal of  the $C^*$-algebra $ C_{P,\pi}^\ast(S,S) $.  

\begin{definition}
 We define  the quotient $C^*$-algebra
   \begin{equation*}
 C^*_{P,\pi,\evtl}(S,S):=C^*_{P,\pi}(S,S)/C^*_{P,\pi,\cpt}(S,S) 
\end{equation*}
and the quotient space 
\[
 C^\ast_{P,\pi,\evtl}(T,S)  :=  C^\ast_{P,\pi}(T,S)  / C^\ast_{P,\pi,\cpt}(T,S)  .
\]
 Under the assumption $T\subseteq S$  the space  $C^*_{P,\pi,\evtl}(T,S)$
is a finitely generated and projective Hilbert $C^*_{P,\pi,\evtl}(S,S)$-module. 
\end{definition}

Now suppose $W$ is a $(d{+}1)$-dimensional, connected, complete, oriented Riemannian   manifold  of   bounded geometry and that $P=\widetilde W$ is a principal $\pi$-space over  {$W$}. 
 Choose a triangulation of bounded geometry of the base (compare
  \cite[5.8, 5.9]{HigsonRoeSurgToAnal1}), and lift to a $\pi$-invariant
triangulation $\Sigma$  of  $\widetilde W$. As before, let $S$ be the
disjoint union of all the $\Sigma_p$.  We have the following
Hilbert-Poincar\'{e} complex  over $ C_{{\widetilde W},\pi,\evtl}^\ast(S,S) $: 
\begin{equation}
   C^*_{{\widetilde W},\pi,\evtl}( \Sigma_0,S )  \xleftarrow{b }
   C^*_{{\widetilde W},\pi,\evtl}( \Sigma_1,S )  \xleftarrow{b } \cdots
   \xleftarrow{b }  C^*_{{\widetilde W},\pi,\evtl}(\Sigma_{d{+}1},S ) ,
\end{equation}
with duality operator obtained from cap product with the fundamental cycle as
before.

\begin{definition}
The \emph{eventual  signature class} of $\widetilde W\to W$ is defined to be 
\[
\sig_{\pi,\evtl}(\widetilde W/W):=\sig_{\pi,\evtl} (\Sigma,S) \in K_{d+1}(
C^*_{{\widetilde W},\pi,\evtl}(S,S) ).
\]
\end{definition}

\begin{theorem}\label{thm:htpyinv}
The eventual   signature class  is invariant under  oriented eventual homotopy
equivalences  which are boundedly controlled and compatible with the
  principal $\pi$-bundles. 
\end{theorem}
\begin{proof}
 This is a special case of \cite[Theorem
          4.3]{HigsonRoeSurgToAnal1} once we have shown that our eventual
          homotopy equivalence $f\colon Z\morphism W$ induces an algebraic
          homotopy 
          equivalence of Hilbert-Poincar\'e complexes in the sense of
          \cite[Definition 4.1]{HigsonRoeSurgToAnal1}. For this, we observe
          that the map $f$ and its eventual homotopy inverse $g\colon
          W\morphism Z$, although only defined outside a compact subset,
          define in the usual way chain maps $f_*,g_*$ between the
          $C^*_{\widetilde W,\pi,\evtl}(S,S)$-Hilbert chain
          complexes of the triangulations, as compactly supported morphisms
          are divided out. For the construction of $f_*$, we need bounded control of the
          maps and bounded geometry of the triangulations. Similarly, the
          eventual homotopies can be used in the standard way to obtain chain
          homotopies between the identity and the composition of these induced maps,
          showing that $f_*$ is a homology isomorphism. Finally, $f$ being
          orientation-preserving, it maps the (locally finite) fundamental cycle of $Z$,
          considered as a cycle relative to a suitable compact subset, to a
          cycle homologous to the fundamental cycle of $W$, again relative to
          a suitable compact subset. As the duality operator of the
          Hilbert-Poincar\'e chain complex $C^*_{\widetilde
            W,\pi,\evtl}(\Sigma_\bullet,S)$, obtained from cap product with the
          fundamental cycle, 
          is determined already by such a relative fundamental cycle, the
          induced map $f_*$ intertwines the two duality operators in the sense
          of \cite[Definition 4.1]{HigsonRoeSurgToAnal1}. Therefore 
          \cite[Theorem 4.3]{HigsonRoeSurgToAnal1} implies the assertion.
\end{proof}

\subsection{Proof of Theorem~\ref{thm-eventual-homotopy-invariance}}

Theorem~\ref{thm-eventual-homotopy-invariance} is an almost immediate corollary of Theorem $\ref{thm:htpyinv}$ above.  We just need the following supplementary computation:

\begin{lemma}
\label{lem-swindle}
Assume that  $W$ is  noncompact.
The kernel of the  homomorphism 
 \[
 K_*( C_{ {\widetilde W},\pi}^\ast(S,S))
 \longrightarrow
 K_*(  C_{ {\widetilde W},\pi,\evtl}^\ast(S,S))
\] 
induced from the quotient map from   $C^*_{{\widetilde W},\pi}(S,S)$ to
$C^*_{ {\widetilde W},\pi,\evtl}(S,S)$  is included in the kernel of the homomorphism
\[
 K_*( C_{ {\widetilde W},\pi}^\ast(S,S))
 \longrightarrow 
 K_*(C^*_{\pi}( {\widetilde W})).
 \]
\end{lemma}

\begin{proof}
Since $W$ is noncompact it contains  a ray $R$ that goes to
  infinity. Let $\widetilde  R\subset \widetilde W$ be its inverse image under the
  projection $\widetilde W\to W$. We have the ideal $C^*_{\pi}(\widetilde  R\subset
  \widetilde W)$ generated by operators supported in bounded neighborhoods of
  $\widetilde  R$. Because every $\pi$-compactly supported operator is also
  supported in a bounded neighborhood of $\widetilde  R$ we get a commutative
  diagram
  \begin{equation}\label{eq:algebra_net}
    \begin{CD}
      C^*_{\widetilde W,\pi,\cpt}(S,S) @>>> C^*_{\widetilde W,\pi}(S,S)
      @>>> C^*_{\widetilde W,\pi,\evtl}(S,S)\\
      @VVV @VVV\\
      C^*_{\pi}(\widetilde  R\subset\widetilde W) @>>> C^*_{\pi}(\widetilde W).
    \end{CD}
  \end{equation}
  We have a canonical isomorphism $K_*(C^*_{\pi}(\widetilde  R))\xrightarrow{\cong}
  K_*(C^*_{\pi}(\widetilde  R\subset\widetilde W))$, see \cite[Proposition
  2.9]{SchickZadeh}, and a standard Eilenberg swindle 
  shows that
  \begin{equation}\label{eq:vanishing}
    K_*(C^*_{\pi}(\widetilde  R\subset\widetilde W)) = 0.
  \end{equation}
See \cite[Section 7, Proposition 1]{HigsonRoeYu93} or \cite[Proposition
2.6]{SchickZadeh}.
Now we apply K-theory to the diagram \eqref{eq:algebra_net}. The long exact
K-theory sequence of the first line of \eqref{eq:algebra_net}
shows that the kernel we have to consider is equal to the image of the map
induced by the first arrow of this line. Naturality and commutativity implies
by \eqref{eq:vanishing} that this image is mapped to zero in
$K_*(C^*_{\pi}(\widetilde W))$, as claimed.
\end{proof}

\begin{proof}[Proof of   Theorem~\ref{thm-eventual-homotopy-invariance}]
If $W$ is compact, then so is $Z$, and the transverse signatures of both are zero.  Otherwise, from Theorem $\ref{thm:htpyinv}$ we have that 
\[ 
\sig_{\pi,\evtl}(\widetilde W / W) =  \sig_{\pi,\evtl}(\widetilde Z/ Z) \in
K_{ {d+1}}( C_{{\widetilde W},\pi,\evtl}^\ast(S,S) ) ,
\]
and from Lemma~\ref{lem-swindle} we conclude from this  that 
\[
\sig_{\pi}(\widetilde W / W)  = \sig_{\pi}(\widetilde Z/ Z)  \in K_{d{+}1}(
C_{\pi}^\ast({\widetilde W})). \]
 The theorem now follows from  Theorem \ref{thm-bounded-homotopy-invariance},
 the appropriate version of  Roe's partitioned manifold index
theorem.\end{proof}

  \begin{remark}
 The proof of the positive scalar curvature result in \cite{HankePapeSchick15}
 directly manipulates the spinor Dirac operator in situations where good
 spectral control is available only outside a compact set. This is parallel to
 our treatment of \emph{eventual} homotopy equivalences. There is a large body
 of work that studies the homotopy invariance of signatures, and consequences,
 using signature operator methods rather than triangulations.  For example the
 results of \cite{HigsonRoeSurgToAnal1,HigsonRoeSurgToAnal2} are analyzed from
 this point of view in \cite{MR3514938}, using methods  in part borrowed from
 \cite{MR1142484, MR3029424}. It seems to be a challenge to prove Theorem
 \ref{thm-main-theorem}, and in particular Theorem
 \ref{thm-eventual-homotopy-invariance}, using signature operator methods.
 But it might be useful to do so in order to study secondary invariants like
 $\rho$-invariants in the situation of Theorem \ref{thm-main-theorem}.
  \end{remark}

\bibliographystyle{abbrv}
\bibliography{ref}

\end{document}